\documentclass{amsart}
\usepackage{amsmath}
\usepackage{amsfonts}
\usepackage{amssymb}
\usepackage{graphicx}
\usepackage{amsthm}
\usepackage{cite}
\usepackage{tikz}

\theoremstyle{plain}
\newtheorem{thm}{Theorem}

\newtheorem{prop}{Proposition}

\theoremstyle{definition}

\theoremstyle{remark}

\newtheorem{case}{Case}

\begin{document}
\title[Complexity of homeomorphism between compact metric spaces]{The complexity of the homeomorphism relation between compact metric spaces}
\author{Joseph Zielinski}
\address{Department of Mathematics, Statistics, and Computer Science,
University of Illinois at Chicago,
322 Science and Engineering Offices (M/C 249),
851 S. Morgan Street,
Chicago, IL 60607-7045}
\email{zielinski@math.uic.edu}
\begin{abstract}
	We determine the exact complexity of classifying compact metric spaces up to homeomorphism. More precisely, the homeomorphism relation on compact metric spaces is Borel bi-reducible with the complete orbit equivalence relation of Polish group actions. Consequently, the same holds for the isomorphism relation between separable commutative C*-algebras and the isometry relation between C(K)-spaces. 
\end{abstract}
\maketitle
\section*{Introduction}
In the present note, we address the following problem, namely, what is the complexity of classifying compact metric spaces up to homeomorphism. We prove, specifically, that it has the same complexity as affine homeomorphism of metrizable Choquet simplexes. As a consequence of recent results of M. Sabok, and earlier work of J. D. Clemens, S. Gao, A. S. Kechris, and the authors of \cite{zbMATH06325399}, this shares the same complexity as isometry of separable, complete metric spaces, isomorphism of separable C*-algebras, and the most complicated of the orbit equivalence relations of Polish group actions, in a manner that we now make precise.

Presented with a family of mathematical objects, much of the challenge involved in acquiring an understanding of the class and its members lies in determining when two arbitrary objects are---or are not---isomorphic in an appropriate sense. That is, the difficulty in understanding the objects in the family often resolves to the difficulty in understanding this isomorphism relation itself. In the abstract, this becomes a particular instance of a more general problem: Given an equivalence relation, $ \mathsf{E} $, on a set, $ X $, how difficult can it be to ascertain when two arbitrary points in $ X $ are $ \mathsf{E} $-related? In the last few decades, there has been a great deal of work done to apprehend this question by considering families of equivalence relations and viewing the classification problem for a given relation relative to others from the same collection.

In general, the family under consideration consists of equivalence relations defined on Polish (separable, completely metrizable) spaces, typically for which the relation is an analytic subset of the product space. Often, depending on the particular relation under consideration, one is interested in some subfamily to which it belongs, e.g., Borel equivalence relations, orbit equivalence relations, or relations whose classes have countably-many members. In this framework, to understand an equivalence relation relative to others is to determine where it is situated in the preorder of ``Borel reducibility'' within its family.

Suppose $ \mathsf{E} $ and $ \mathsf{F} $ are equivalence relations on Polish spaces $ X $ and $ Y $, respectively. Write $ \mathsf{E} \leqslant_{B} \mathsf{F} $ ($ \mathsf{E} $ is \textit{Borel reducible} to $ \mathsf{F} $) when there is a Borel-measurable function $ f: X \to Y $ satisfying $ x \mathsf{E}y \iff f(x) \mathsf{F} f(y) $ for all $ x,y \in X $. When also $ \mathsf{F} \leqslant_{B} \mathsf{E} $, they are \textit{Borel bi-reducible}, written $ \mathsf{E} \sim_{B} \mathsf{F} $.
The preorder, $ \leqslant_{B} $, is itself immensely complicated, even when restricted to the subfamilies mentioned above. For more details on the general theory, see \cite{MR2455198} and \cite{MR1725642}.

Often, a ``naturally-occurring'' classification problem admits an analysis in this framework. That is, there is a Polish space that can be viewed as representing a class of mathematical objects, for which a natural notion of isomorphism in that class is an analytic equivalence relation.

For example, it was shown in \cite{MR1950332} and \cite{MR2926279} that the isometric classification of Polish metric spaces is \textit{complete} in the class of those relations reducible to the orbit equivalence relation of a Polish group action (that is, it is a member of the class, and $ \leqslant_{B} $-above every other member). Later, J. Melleray in \cite{MR2308492} demonstrated that the classification of separable Banach spaces by linear isometries is also complete in this class. More recently, in \cite{MR3176617} and \cite{zbMATH06325399}, the isomorphism relation on separable C*-algebras was seen to be a member of this family as well, and it too was proven to be complete in \cite{sabok2013completeness}.

Similarly, the homeomorphic classification of compact metric spaces has been known to also be a member of the above, and, as shown in \cite{MR1725642}, is strictly $ \leqslant_{B} $-above any continuous action of the infinite symmetric group, which aside from the complete relations, are the most complicated in the class for which there presently are natural representatives of known complexity. However, it has remained an open question as to whether homeomorphism of compact metric spaces is complete (e.g., \cite{MR1950332} Problem 10.3). The main result here is that it is indeed a complete relation.
\begin{thm} \label{universal homeomorphism}
Every orbit equivalence relation of a Polish group action is Borel reducible to the homeomorphism relation on compact metric spaces.
\end{thm}

The proofs below consist of a preliminary construction followed by a chain of reductions, beginning with the relation of affine homeomorphism of Choquet simplexes, and culminating in the homeomorphism relation for compact spaces. As the former is a complete relation, it then follows that they are are bi-reducible. Additionally, while we noted above that linear isometry of separable Banach spaces and isomorphism of separable C*-algebras were shown to have the same complexity as a complete orbit equivalence relation in \cite{MR2308492} and \cite{sabok2013completeness}, in the latter case it was established that this complexity is achieved on the subclass of simple AI algebras.
Similarly, it will be a corollary of our result (as observed in \cite{MR1950332} Ch.\ 10, and \cite{MR3176617}) that these are still complete relations when restricted to the respective subclasses of C(K)-spaces and commutative C*-algebras.

\subsection*{Acknowledgments}
I am grateful to my advisor, Christian Rosendal, for many enlightening conversations, and to Marcin Sabok and Julien Melleray for their helpful comments on earlier drafts of this paper.

\subsection*{Preliminaries}
For a Polish space, $ X $, let $ K(X) = \{A \subseteq X \mid A \text{ is compact} \} $ denote the hyperspace of compact subsets of $ X $. We endow $ K(X) $ with the \textit{Vietoris topology}, a Polish topology induced by the Hausdorff metric, in which $ K(X) $ is compact whenever $ X $ is (see \cite{MR1321597} 4.F).
Let $ \mathcal{Q} = [0,1]^{\mathbb{N}} $ denote the Hilbert cube. Every separable metric space embeds homeomorphically into $ \mathcal{Q} $, so in particular $ \mathcal{Q} $ contains a homeomorphic copy of every compact metric space, and we view $ K(\mathcal{Q}) $ as the space of metrizable compact spaces.

We will, at a later stage, need the notion of a Z-set, specifically the Z-sets in $ \mathcal{Q} $. The Z-sets of a topological space form an ideal of closed sets, and are central to the study of infinite-dimensional topology. We will not need a precise definition for what follows, but will mention here some important features. First, one can embed $ \mathcal{Q} $ into itself so that it (and, consequently, its closed subsets) are Z-sets. Second, any homeomorphism $ A \to B $ between Z-sets extends to a homeomorphism $ \mathcal{Q} \to \mathcal{Q} $. For more details, see chapter 5 of \cite{MR1851014}. We note that it is a direct consequence of these facts (originally observed by Kechris and S. Solecki) that the homeomorphism relation for compact metric spaces Borel-reduces to an orbit equivalence relation of a Polish group action. Indeed, the above embedding, $ \theta $, of $ \mathcal{Q} $ as a Z-set in $ \mathcal{Q} $, induces a map $ \theta^{*}:K(\mathcal{Q}) \to K(\mathcal{Q}) $, which is a reduction of the homeomorphism relation between compact subsets of $ \mathcal{Q} $ to the orbit equivalence relation of the shift action of $ \operatorname{Homeo}(\mathcal{Q}) $ on $ K(\mathcal{Q}) $.
Namely, two sets $ A $ and $ B $ are homeomorphic if and only if their images, $ \theta^{*}(A) $ and $ \theta^{*}(B) $, are homeomorphic, which in turn occurs if and only if---by the extension property---it can be witnessed by an autohomeomorphism of $ \mathcal{Q} $.

It was shown in \cite{MR1425877} that for any Polish group, $ G $, there is a complete orbit equivalence relation for continuous actions of $ G $, and that for universal Polish groups this equivalence relation will, in fact, be complete in the entire class of orbit equivalence relations of Polish group actions. We fix such a complete orbit equivalence relation now, and denote it by $ \mathsf{E}_{\mathrm{grp}} $.

Finally, for a homeomorphism $ f:X \to Y $, we let $ f^{n} $ denote the map, $ f^{n} = (f \times f \times \cdots \times f): X^{n} \to Y^{n} $.

\section*{Homeomorphism of compact metric spaces}
Suppose $ X $ is a compact metric space, and $ A $ a closed subset containing all isolated points of $ X $. Let $ I(X,A) \subseteq X \times [0,1] $ be constructed as follows: Let $ (d_{n})_{n \in \mathbb{N}} $ enumerate, with infinite repetition, a countable dense subset of $ A $. For each $ n $, set $ \widetilde{a}_{n} = (d_{n},\frac{1}{n}) $, and set $ I(X,A) = (X \times \{0\} ) \cup \{\widetilde{a}_{1},\widetilde{a}_{2},...\} $. Then $ I(X,A) $ is a closed subset of $ X \times [0,1] $ and as such is compact, and we identify $ X $ with $ X \times \{0\} $. We may view $ I(X,A) $ as the space $ X $, together with a new sequence of isolated points accumulating on $ A $. In the sequel, whenever this construction is invoked, we will denote these new isolated points with tildes as we have above. There may appear to be an ambiguity, as the definition of the $ \widetilde{a}_{n} $---and therefore of $ I(X,A) $---requires making a choice in $ (d_{n})_{n \in \mathbb{N}} $, our enumeration of a countable dense subset of $ A $. However, it will follow from the proof below that any two such choices result in the same space, up to homeomorphism.

\begin{prop} \label{isolated points}
Let $ X,Y $ be compact metric spaces, with $ A,B $ closed subsets so that $ A $ (resp. $ B $) contains all isolated points of $ X $ (resp. $ Y $). Let $ I(X,A) $ and $ I(Y,B) $ be constructed as above. Then every homeomorphism $ g: X \to Y $ with $ g[A] = B $ extends to a homeomorphism $I(X,A) \to I(Y,B) $. Conversely, if $f: I(X,A) \to I(Y,B) $ is a homeomorphism, then $ f[X] = Y $ and $ f[A] = B $.
\end{prop}

\begin{proof}

Suppose $f: I(X,A) \to I(Y,B)$ is a homeomorphism. Since we have assumed that all isolated points of $ X $ are in $ A $, they are now limits of the $ \widetilde{a}_{n} $, and are no longer isolated. So as $ f $ must send isolated points to isolated points and non-isolated points to non-isolated points, $ f $ restricts to a homeomorphism $ X \to Y $, and $ f[\{\widetilde{a}_{1},\widetilde{a}_{2},...\}] = \{\widetilde{b}_{1},\widetilde{b}_{2},... \} $. So then
\begin{align*}
f\left[A \cup \{\widetilde{a}_{1},\widetilde{a}_{2},...\} \right] &= f\left[\overline{\{\widetilde{a}_{1},\widetilde{a}_{2},...\} } \right]\\
&= \overline{\{\widetilde{b}_{1},\widetilde{b}_{2},...\} }\\
&= B \cup \{\widetilde{b}_{1},\widetilde{b}_{2},...\}  
\end{align*}
and so $ f[A] = B $.

Conversely, suppose $ g:X \to Y $ is a homeomorphism with $ g[A] =B $. We extend $ g $ to $ f:I(X,A) \to I(Y,B) $ via a back-and-forth. At odd stages $ k $, if a bijection has been constructed between finite subsets $ \{\widetilde{a}_{n_{1}},\widetilde{a}_{n_{2}},...,\widetilde{a}_{n_{k-1}}\} $ and $ \{\widetilde{b}_{m_{1}},\widetilde{b}_{m_{2}},...,\widetilde{b}_{m_{k-1}}\} $, we extend the domain to include $ \widetilde{a}_{n_{k}} $, where $ n_{k} $ is the least not already among the $ n_{i} $ for which the bijection is defined. Recall that each $ \widetilde{a}_{n} $ is of the form $(d_{n},\frac{1}{n}) $ for a countable dense subset $ (d_{n}) $ of $ A $. Likewise, each $ \widetilde{b}_{m} $ has the form $ (e_{m},\frac{1}{m}) $. So, choose $ \widetilde{b}_{m_{k}} $ to be the least element not already in the range of the bijection, and for which $ \rho_{Y}(e_{m_{k}},g(d_{n_{k}})) < \frac{1}{k}$, where $ \rho_{Y} $ is a fixed compatible metric for $ Y $. At even stages, reverse the roles.

Now extend $ g $ to $ f: I(X,A) \to I(Y,B) $ according to the above construction. We now show that for any $ x \in I(X,A) $ and open neighborhood $ \mathcal{V} $ of $ f(x) $, there is an open neighborhood, $ \mathcal{U} $ of $ x $ with $ f[\mathcal{U}] \subseteq \mathcal{V} $. We may assume $ x \in X $, for if $ x $ is among the $ \widetilde{a}_{n} $, the set $ \mathcal{U} = \{x\} $ is as desired.

Pick $ M \in \mathbb{N} $ so that the basic open set $I(Y,B) \cap [B_{Y}(g(x),\frac{2}{M}) \times [0,\frac{1}{M})] \subseteq \mathcal{V} $. Let $ W \subseteq X $ be $ W = g^{-1}[B_{Y}(g(x),\frac{1}{M})] $ and let $ N \in \mathbb{N} $ be large enough that $ B_{X}(x,\frac{2}{N}) \subseteq W $.

Set $ K_{1} $ to be the largest index $ n $ such that $ \widetilde{a}_{n} \mapsto \widetilde{b}_{m} $ for some $ m \leq M $, set $ K_{2} $ to be the largest $ n $ with $ \widetilde{a}_{n} $ in the domain of the partial bijection constructed by stage $ M $, and set $ K_{3} $ to be the largest $ n $ with $ \widetilde{a}_{n} $ in the domain by stage $ N $. Let $ K = \max\{K_{1},K_{2},K_{3}\}+1 $. We claim that if $ y \in \mathcal{U} = I(X,A) \cap [B_{X}(x,\frac{1}{N}) \times [0,\frac{1}{K})]  $, then $ f(y) \in \mathcal{V} $.

\begin{case}[$ y \in X $]
If $ y \in X $, then $ y \in W $, and so $ f(y) \in f[W \times \{0\}] = B_{Y}(g(x),\frac{1}{M}) \times \{ 0 \} \subseteq \mathcal{V}  $.
\end{case}

\begin{case}[odd stages]
Suppose $ y = \widetilde{a}_{n} $ and $ f(y) = \widetilde{b}_{m} $ for some $ n, m $, and that moreover the assignment $ \widetilde{a}_{n} \mapsto \widetilde{b}_{m} $ was constructed at an odd stage of the back-and-forth. We show $ \widetilde{b}_{m} = (e_{m}, \frac{1}{m}) \in B_{Y}(g(x),\frac{2}{M}) \times [0,\frac{1}{M}) $.  Since $ \widetilde{a}_{n} = (d_{n},\frac{1}{n}) \in \mathcal{U} $, $ \frac{1}{n} < \frac{1}{K} $, and so $ n > K $. Now since $ d_{n} \in W $, $ \rho_{Y}(g(x),g(d_{n})) < \frac{1}{M} $, and since $ n > K_{2} $, the bijection was constructed after stage $ M $, and so $ \rho_{Y}(e_{m},g(d_{n})) < \frac{1}{M} $. Then $ \rho_{Y}(g(x),e_{m}) \leq \rho_{Y}(g(x),g(d_{n})) + \rho_{Y}(g(d_{n}),e_{m}) < \frac{2}{M} $. Also, since $ n > K_{1} $, $ m > M $ and $ \frac{1}{m} < \frac{1}{M} $. So $ \widetilde{b}_{m} \in \mathcal{V} $.
\end{case}

\begin{case}[even stages]
If $ y = \widetilde{a}_{n} $, $ f(y)=\widetilde{b}_{m} $, and the construction was at an even stage, then since $ n > K_{3} $, they were added to the bijection at stage greater than $ N $. Therefore, $ \rho_{X}(d_{n},g^{-1}(e_{m})) < \frac{1}{N} $. So $ \rho_{X}(x,g^{-1}(e_{m})) \leq \rho_{X}(x,d_{n}) + \rho_{X}(d_{n},g^{-1}(e_{m})) < \frac{2}{N}$. Then $ g^{-1}(e_{m}) \in W $, and so $ e_{m} \in B_{Y}(g(x),\frac{1}{M}) $. On the other hand, since $ n > K_{1} $, again $ \frac{1}{m} < \frac{1}{M} $, so $ \widetilde{b}_{m} \in \mathcal{V} $.
\end{case}

So $ f $ is a continuous bijection between compact Hausdorff spaces, and so is a homeomorphism.
\end{proof}

It is worth observing that the requirement that $ A $ contains the isolated points of $ X $ is satisfied trivially in the cases where $ X $ is perfect, or when $ X = A $. But in the first, ``$ X $ has no isolated points'' could easily be replaced by some other topological feature, like, ``$ X $ is path connected, and remains so with the removal of a single point''. In this way, one can code any finite or even countable ordered sequence of subsets of $ X $ by embedding it as a Z-set into $ \mathcal{Q} $ (which has the above property), and marking it and its subsets off as limits of one-dimensional stars of differing valence and vanishing diameter. We will forgo giving the details here, as the simpler version above is sufficient for the constructions of Proposition \ref{countable up to permutation} and Theorem \ref{universal homeomorphism}.

In what follows, it is a routine matter to check that the reductions at each stage are Borel-measurable, and these verifications will be omitted. However, as an illustration, let us show here how we may take the above function, $ I $, to be a Borel map from $ \{(X,A) \in K(\mathcal{Q}) \times K(\mathcal{Q}) \mid A \subseteq X \} $ into $ K(\mathcal{Q} \times [0,1]) $. Fix $ h: \mathbb{N} \to \mathbb{N} $, an enumeration with infinite repetition, and let $ \{s_{n}\}_{n \in \mathbb{N}} $ be the Kuratowski-Ryll-Nardzewski selector functions $ K(\mathcal{Q}) \to \mathcal{Q} $ as in \cite{MR1321597}, Theorem 12.13. Define $ I(X,A) $ as before, taking $ d_{n} = s_{h(n)}(A) $. Then for any open $ U $ in $ \mathcal{Q} \times [0,1] $, $ I(X,A) \cap U \neq \emptyset $ if and only if $ \exists n, (s_{n}(X),0) \in U $ or $ \exists n, (s_{h(n)}(A),\frac{1}{n}) \in U $, a Borel condition by the measurability of the selector functions.\newline

Consider the space $ \{(X,R) \in K(\mathcal{Q}) \times K(\mathcal{Q}^{3}) \mid R \subseteq X^{3} \} $, and let $ \cong_{(3)} $ denote the equivalence relation where $ (X,R) \cong_{(3)} (Y,S) $ if and only if there is a homeomorphism $f: X \to Y $ with $ f^{3}[R] = S $.

\begin{prop} \label{universality of (3)}
$ \mathsf{E}_{\mathrm{grp}} \leqslant_{B} \; \cong_{(3)} $.
\end{prop}

\begin{proof}
Recall that a metrizable, compact, convex subset of a locally convex space is a \textit{metrizable Choquet simplex} if every point is the barycenter of a unique probability measure supported on the extreme boundary (for details, see \cite{MR1863688}, Ch.\ 15). Let $ \mathbb{K}_{\mathrm{Choq}} $ denote the convex subsets of $ \mathcal{Q} $ that are Choquet simplexes. In Section 4 of \cite{MR3176617}, it is shown that $ \mathbb{K}_{\mathrm{Choq}} $ is a Borel subset of $ K(\mathcal{Q}) $, and may be taken to be the standard Borel space of metrizable Choquet simplexes. In \cite{sabok2013completeness}, Sabok shows that $ \mathsf{E}_{\mathrm{grp}} \sim_{B} \; \approx_{a} $, the relation of affine homeomorphism on $ \mathbb{K}_{\mathrm{Choq}} $. Define $ \Gamma : \mathbb{K}_{\mathrm{Choq}} \to K(\mathcal{Q}^{3}) $ by \[ \Gamma(S) = \{(x,y,z) \in S^{3} \mid \frac{1}{2}x + \frac{1}{2}y = z \}. \]
Then we claim $ S \mapsto (S, \Gamma(S)) $ witnesses the reduction $ \approx_{a} \; \leqslant_{B} \; \cong_{(3)} $.  Clearly if $ S \approx_{a} T $, and $ f: S \to T $ is an affine homeomorphism, then  $ \frac{1}{2}x + \frac{1}{2}y = z $ iff $ \frac{1}{2}f(x) + \frac{1}{2}f(y) = f(z) $, and so $ f^{3}[\Gamma(S)] = \Gamma(T) $, and so $ (S,\Gamma(S)) \cong_{(3)} (T, \Gamma(T)) $.

Conversely, suppose $ f $ is a homeomorphism of $ S $ and $ T $ with $ f^{3}[\Gamma(S)] = \Gamma(T) $, $ x, y \in S $, and $ \lambda \in (0,1) $. We must show that $ f(\lambda x + (1- \lambda) y) = \lambda f(x) + (1-\lambda  ) f(y) $. The case where $ \lambda = \frac{1}{2} $ is already done, and by induction, so too are the cases where $ \lambda $ has the form $ \frac{m}{2^{n}} $. 
For other $ \lambda $, let $ \lambda_{k} $ be a sequence of dyadic rationals with $ \lambda_{k} \to \lambda $. Then  \[ \lambda_{k} x + (1-\lambda_{k})y \to \lambda x + (1-\lambda)y  \]
and by continuity,  \[ \lambda_{k} f(x) + (1-\lambda_{k})f(y) = f(\lambda_{k} x + (1-\lambda_{k})y) \to f(\lambda x + (1-\lambda)y)   \]
while  \[ \lambda_{k} f(x) + (1-\lambda_{k})f(y) \to \lambda f(x) + (1-\lambda)f(y) \] 
so $ f(\lambda x + (1-\lambda)y) = \lambda f(x) + (1-\lambda)f(y) $. So $ f $ is an affine homeomorphism and $ S \approx_{a} T $.
\end{proof}

Given $ \vec{A} = (A_{1},A_{2},...)$ and $ \vec{B} = (B_{1},B_{2},...) $ in $ K(\mathcal{Q})^{\mathbb{N}} $, say that $ \vec{A} \cong_{\mathrm{perm}} \vec{B} $ if and only if there is an $ f \in \operatorname{Homeo}(\mathcal{Q}) $ and $ \sigma \in S_{\infty} $ such that $ f[A_{n}] = B_{\sigma(n)} $ for all $ n $.
Now, for a pair of compact sets $ (X,R) $, where $ R \subseteq X^{3} $, define
\begin{align*}
\widetilde{X} &= I(X,X) = X \cup \{ \widetilde{a}_{1}, \widetilde{a}_{2},... \} \\
B_{n} &= \{\widetilde{a}_{n}\} \times \widetilde{X}^{2}\\
C_{n} &= \widetilde{X} \times \{\widetilde{a}_{n}\} \times \widetilde{X}\\
D_{n} &= \widetilde{X}^{2} \times \{\widetilde{a}_{n}\}\\
E_{n} &= B_{n} \cup C_{n}\\
F_{n} &= B_{n} \cap D_{n} 
\end{align*}
Let $ \Psi: (X,R) \mapsto (\widetilde{X}^{3},R,B_{1},C_{1},D_{1},E_{1},F_{1},B_{2},C_{2},...) $. Fix an embedding, $ \imath $, of $ (\mathcal{Q} \times [0,1])^{3} $ into $ \mathcal{Q} $ as a Z-set, and let $ (\imath^{*})^{\mathbb{N}}: K((\mathcal{Q} \times [0,1])^{3})^{\mathbb{N}} \to K(\mathcal{Q})^{\mathbb{N}} $ be the induced map. Set $ \Phi = (\imath^{*})^{\mathbb{N}} \circ \Psi $, but for convenience identify a set in $ \Phi(X,R) $ with the corresponding set in $ \Psi(X,R) $. 

\begin{prop} \label{countable up to permutation}
$ \Phi $ is a reduction from $ \cong_{(3)} $ to $ \cong_{\mathrm{perm}} $.
\end{prop}

\begin{proof}
Suppose $ (X,R) \cong_{(3)} (Y,S) $, witnessed by $ f $. To fix notation, let $ \widetilde{Y} $ = $ Y \cup \{\widetilde{b}_{1},\widetilde{b}_{2},... \} $, and let $ H_{m},...,L_{m} $ denote $ Y $'s versions of $ B_{n},..., F_{n} $, respectively. First, extend $ f $ to a homeomorphism $ \widetilde{f}: \widetilde{X} \to \widetilde{Y} $ as in Proposition \ref{isolated points}. Let $ \tau \in S_{\infty}  $ be $ \tau(n) = m $ iff $ \widetilde{a}_{n} \mapsto \widetilde{b}_{m} $, and let $ \sigma \in S_{\infty} $ be $ \sigma(1) = 1 $, $ \sigma(2) = 2 $, and for $ k\in \mathbb{N} $ and $ 0 \leq j \leq 4 $, $ \sigma(5k+j-2) = 5\tau(k) +j -2 $. In words, $ \sigma $ fixes $ 1 $ and $ 2 $ and otherwise permutes quintuplets of integers according to the back-and-forth.

Now $ \widetilde{f} $ determines a homeomorphism $ \widetilde{f}^{3}: \widetilde{X}^{3} \to \widetilde{Y}^{3} $, which by assumption has $ \widetilde{f}^{3}[R] = f^{3}[R] = S $, and which by construction has $ \widetilde{f}^{3}[B_{n}] = \widetilde{f}^{3}[\{\widetilde{a}_{n}\} \times \widetilde{X}^{2}] = \{\widetilde{b}_{\tau(n)}\} \times \widetilde{X}^{2} = H_{\tau(n)} $ and similarly $ \widetilde{f}^{3}[C_{n}] = I_{\tau(n)} $,...,$ \widetilde{f}^{3}[F_{n}] = L_{\tau(n)} $. So as $ \widetilde{X}^{3} $ is embedded as a Z-set in $ \mathcal{Q} $, extend $ \widetilde{f}^{3} $ to some $ \widehat{f} \in \operatorname{Homeo}(\mathcal{Q}) $. Then $ \widehat{f} $ and $ \sigma $ witness $ (\widetilde{X}^{3},R,B_{1},C_{1},...) \cong_{\mathrm{perm}} (\widetilde{Y}^{3},S,H_{1},I_{1},...) $.

Conversely, assume $ \Phi(X,R) \cong_{\mathrm{perm}} \Phi(Y,S) $ witnessed by $ f \in \operatorname{Homeo}(\mathcal{Q}) $ and $ \sigma \in S_{\infty} $. For $ A,A' \subseteq \mathcal{Q} $, $ A \subseteq A' $ iff $ f[A] \subseteq f[A'] $, and consequently, we may recover information about $ \sigma $ based on how many and which sets a given one contains, or is contained in. First, since every other set in the enumeration of $ \Phi(X,R) $ is a proper subset of the first, $ \widetilde{X}^{3} $, and this is also true of $ \Phi(Y,S) $, we see that $ f[\widetilde{X}^{3}] = \widetilde{Y}^{3} $ (i.e., $ \sigma(1) = 1 $). Among the remaining sets in $ \Phi(X,R) $, $ R $ is the only one that is disjoint from every other, and so $ f[R] = S $ ($ \sigma(2) = 2 $).

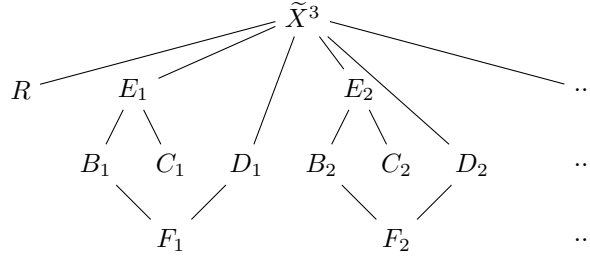
\begin{figure}[h]
\centering
\begin{tikzpicture}[level distance=1cm,level 2/.style={sibling distance=10mm}]
\node (X) {$ \widetilde{X}^{3} $}
	child{ node (R) {$ R $} }
	child{node (E1) { $ E_{1} $}
		child{node (B1) {$ B_{1} $}}
		child{node (C1) {$ C_{1} $}
				child{node (F1) {$ F_{1} $} edge from parent[draw=none]}
			}
		}
	child{
		child{node (D1) {$ D_{1} $} edge from parent[draw=none] }
		edge from parent[draw=none]}
	child{node (E2) {$ E_{2} $}
		child{node (B2) {$ B_{2} $}}
		child{node (C2) {$ C_{2} $}
				child{node (F2) {$ F_{2} $} edge from parent[draw=none] }
			}
		}
	child{
		child{node (D2) {$ D_{2} $} edge from parent[draw=none]}
		edge from parent[draw=none]}
	child{ node {$ ... $}
	child{node {$ ... $}edge from parent [draw=none]
	child{node {$ ... $}edge from parent [draw=none]}}
	}
	
;
\draw (X) -- (D1);
\draw (X) -- (D2);
\draw (B1) -- (F1);
\draw (D1) -- (F1);
\draw (B2) -- (F2);
\draw (D2) -- (F2);
\end{tikzpicture}
\caption{Set containment in $ \Phi(X,R) $}

\end{figure}

So setting aside $ \widetilde{X}^{3} $ and $ R $, for each $ n $, the sets with index $ n $ are distinguished from each other in a similar fashion. For instance, $ F_{n} $ contains no set, and is contained in three. Since $ f[F_{n}] $ must also possess this property, $ f[F_{n}] = L_{m} $ for some $ m $. Likewise, sets labeled with $ B $ are sent to sets labeled with $ H $, $ C $ to $ I $, $ D $ to $ J $, and $ E $ to $ K $.  Let $ \tau $ denote the permutation given by $ f[C_{n}] = I_{\tau(n)} $.

But $ C_{n} \subseteq E_{n} $ requires that $ I_{\tau(n)} \subseteq f[E_{n}] $ and the latter, by the above, is labeled with $ K $. So $ f[E_{n}] = K_{\tau(n)} $.   Similar considerations show that $ f[B_{n}] = H_{\tau(n)} $, $ f[F_{n}] = L_{\tau(n)} $, and $ f[D_{n}] = J_{\tau(n)} $. So, as in the forward direction, $ \sigma $ fixes $ 1 $ and $ 2 $, and otherwise permutes ordered blocks of five according to $ \tau $.

Now, for every $ n $, $ (\widetilde{a}_{n},\widetilde{a}_{n},\widetilde{a}_{n}) $ is the unique element of $ B_{n} \cap C_{n} \cap D_{n} $, and so $ f(\widetilde{a}_{n},\widetilde{a}_{n},\widetilde{a}_{n}) $ is the unique element of $ H_{\tau(n)} \cap I_{\tau(n)} \cap J_{\tau(n)} $, that is, $f(\widetilde{a}_{n},\widetilde{a}_{n},\widetilde{a}_{n}) = (\widetilde{b}_{\tau(n)},\widetilde{b}_{\tau(n)},\widetilde{b}_{\tau(n)}) $. These points are dense in the diagonals of $ \widetilde{X}^{3} $ and $ \widetilde{Y}^{3} $, respectively, and so $ f[\Delta_{\widetilde{X}^{3}}] = f[\overline{\{(\widetilde{a}_{n},\widetilde{a}_{n},\widetilde{a}_{n}) \mid n \in \mathbb{N}  \}}] = \overline{\{(\widetilde{b}_{m},\widetilde{b}_{m},\widetilde{b}_{m}) \mid m \in \mathbb{N}  \}} = \Delta_{\widetilde{Y}^{3}}$. This determines a homeomorphism $ g: \widetilde{X} \to \widetilde{Y} $ by $ g(x) = y $ iff $ f(x,x,x) = (y,y,y) $.

We claim that $ g^{3} = f \upharpoonright \widetilde{X}^{3} $. Note for every $ n $, $ g(\widetilde{a}_{n}) = \widetilde{b}_{\tau(n)} $, and each triple $ (\widetilde{a}_{n},\widetilde{a}_{m},\widetilde{a}_{k}) $ is the unique point in the singleton $ B_{n} \cap C_{m} \cap D_{k} $. So,
\begin{align*}
f[\{(\widetilde{a}_{n},\widetilde{a}_{m},\widetilde{a}_{k})\}] &= f[B_{n} \cap C_{m} \cap D_{k}] = H_{\tau(n)} \cap I_{\tau(m)} \cap J_{\tau(k)}\\
&= \{(\widetilde{b}_{\tau(n)},\widetilde{b}_{\tau(m)},\widetilde{b}_{\tau(k)})  \} = g^{3}[\{(\widetilde{a}_{n},\widetilde{a}_{m},\widetilde{a}_{k})\}]
\end{align*}
and $ f $ and $ g^{3} $ agree on a dense subset of $ \widetilde{X}^{3} $, so are equal. Moreover, $ \widetilde{X} = I(X,X) $, and so by Proposition \ref{isolated points}, $ g $ restricts to a homeomorphism $ X \to Y $, and $ g^{3}[R] = f[R] = S $. So $ g \upharpoonright X $ witnesses $ (X,R) \cong_{(3)} (Y,S) $.
\end{proof}

Consider the space whose members are triples of compact sets, $ (X,B,A) $, such that $ X $ is perfect and $ A \subseteq B \subseteq X $. Let $ \cong_{(1,1)} $ denote the equivalence relation where $ (X,B,A) \cong_{(1,1)} (Y,D,C) $ if and only if there exists a homeomorphism $ f: X \to Y $ such that $ f[A] = C $ and $ f[B] = D $.

\begin{prop} \label{two subsets}
$ \cong_{\mathrm{perm}} \; \leqslant_{B} \; \cong_{(1,1)} $.
\end{prop}

\begin{proof}
Let $ \mathbb{X} = \{(x,y) \in \mathcal{Q}^{2} \mid \forall m \neq n, y_{m} = 0 \vee y_{n} = 0 \} $. That is, $ \mathbb{X} $ consists of those points for which $ y $ has at most one nonzero coordinate. Identify $ \mathcal{Q}  $ with $ \mathcal{Q} \times \{\vec{0}\} $. For $ \vec{A} = (A_{1},A_{2},...) \in K(\mathcal{Q})^{\mathbb{N}} $, let $ \mathbb{A} = \{(x,y) \in \mathbb{X} \mid \forall n, y_{n} = 0 \vee x \in A_{n}  \} $. Let $ \Phi: \vec{A} \mapsto (\mathbb{X}, \mathbb{A}, \mathcal{Q})  $.

Suppose $ \vec{A} \cong_{\mathrm{perm}} \vec{B} $ via $ f $ and $ \sigma $. Let $ h_{\sigma} $ denote the homeomorphism of $ \mathcal{Q} $ taking $ (y_{1},y_{2},...) \mapsto (y_{\sigma(1)}, y_{\sigma(2)},...) $. Consider the homeomorphism $ f \times h_{\sigma^{-1}}:\mathcal{Q}^{2} \to \mathcal{Q}^{2} $. Since, for any $ (x,y) $, $ y $ has at most one nonzero coordinate iff $ h_{\sigma^{-1}}(y) $ has at most one nonzero coordinate, we have that $ (f \times h_{\sigma^{-1}})[\mathbb{X}] = \mathbb{X} $. Also,
\begin{align*}
	(x,y) \in \mathbb{A} & \text{ iff } \forall n, x \in A_{n} \vee y_{n} = 0 \\
	&\text{ iff } \forall n, f(x) \in f[A_{n}] = B_{\sigma(n)} \vee (h_{\sigma^{-1}}(y))_{\sigma(n)} = y_{n} = 0\\
	&\text{ iff } \forall m, f(x) \in B_{m} \vee h_{\sigma^{-1}}(y)_{m} = 0\\
	&\text{ iff }(f \times h_{\sigma^{-1}} )(x,y) \in \mathbb{B}
\end{align*}
showing $ (f \times h_{\sigma^{-1}}) [\mathbb{A}] = \mathbb{B} $. Finally, since $ y = \vec{0} $ iff $ h_{\sigma^{-1}}(y) = \vec{0} $, $ (f \times h_{\sigma^{-1}})[\mathcal{Q}] = \mathcal{Q} $. So the restriction $ (f \times h_{\sigma^{-1}})\upharpoonright \mathbb{X} $ witnesses $ (\mathbb{X},\mathbb{A},\mathcal{Q}) \cong_{(1,1)}(\mathbb{X},\mathbb{B},\mathcal{Q}) $.

So suppose there is a homeomorphism $ g:\mathbb{X} \to \mathbb{X} $ with $ g[\mathbb{A}] = \mathbb{B} $ and $ g[\mathcal{Q}] = \mathcal{Q} $. Since $ g $ fixes $ \mathcal{Q} $ set-wise, it must also fix $ \mathbb{X} \setminus \mathcal{Q} $ set-wise. Now, we may write $ \mathbb{X} \setminus \mathcal{Q} $ as the disjoint union of sets $ \mathcal{X}_{n} = \{(x,y) \in \mathbb{X} \mid x \in \mathcal{Q} \text{ and } y_{n} > 0  \} $. Each $ \mathcal{X}_{n} $, being homeomorphic to $ \mathcal{Q} \times (0,1] $, is path-connected, but no path exists within $ \mathbb{X} \setminus \mathcal{Q} $ connecting distinct $ \mathcal{X}_{n} $ and $ \mathcal{X}_{m} $. Therefore, $ g $ induces a permutation $ \sigma $, by $ \sigma(n) = m $ iff $ g[\mathcal{X}_{n}] = \mathcal{X}_{m} $. Let $ \mathcal{A}_{n} = \mathcal{X}_{n} \cap \mathbb{A} $ and $ \mathcal{B}_{n} = \mathcal{X}_{n} \cap \mathbb{B} $. Then $ g[\mathcal{A}_{n}] = \mathcal{X}_{\sigma(n)} \cap \mathbb{B} = \mathcal{B}_{\sigma(n)} $.

Now, $ \mathcal{A}_{n} = \{(x,y) \mid x \in A_{n} \text{ and } y_{n} > 0  \} $. So since $ A_{n} $ is closed, if $ (x^{k},y^{k})_{k \in \mathbb{N}} $ is a sequence of points in $ \mathcal{A}_{n} $ converging to $ (x,y)  \in \mathbb{X}$, then $ x \in A_{n} $, $ y_{n} \in [0,1] $, and $ y_{m} = 0 $ for $ m \neq n $. On the other hand, given $ x \in A_{n} $, $ (x,(0,0,...,0,\frac{1}{k},0,...)) $ is a sequence in $ \mathcal{A}_{n} $ converging to $ (x,\vec{0}) $. Therefore, $ \overline{\mathcal{A}_{n}} = \mathcal{A}_{n} \cup (A_{n} \times \{\vec{0}\}) $. This holds similarly for the $ \mathcal{B}_{n} $. Identifying $ A_{n} $ and $ B_{n} $ with their copies in $ \mathcal{Q} = \mathcal{Q} \times \{\vec{0}\} $, for every $ n $, $ g[A_{n}] = g[\mathcal{Q} \cap \overline{\mathcal{A}_{n}}] = g[\mathcal{Q}] \cap g[\overline{\mathcal{A}_{n}}] = \mathcal{Q} \cap \overline{\mathcal{B}_{\sigma(n)}} = B_{\sigma(n)} $.
So $ g \upharpoonright \mathcal{Q} $ and $ \sigma $ witness $ \vec{A} \cong_{\mathrm{perm}} \vec{B} $.
\end{proof}

Let $ \approx $ denote the homeomorphism relation on compact metric spaces.

\begin{proof}[Proof of Theorem \ref{universal homeomorphism}]
By Propositions \ref{universality of (3)}, \ref{countable up to permutation}, and \ref{two subsets}, $ \mathsf{E}_{\mathrm{grp}} \leqslant_{B} \; \cong_{(1,1)} $, the latter defined on the class of triples, $ (X,B,A) $, where $ X $ is perfect and $ A \subseteq B \subseteq X $. Let $ I_{2}(X,B,A) = I(I(X,A),B \cup \{\widetilde{a}_{1},\widetilde{a}_{2},...\}) $. That is, $ I_{2}(X,B,A) $ is obtained as in Proposition \ref{isolated points} by first adding a set of isolated points,$ \{\widetilde{a}_{1},\widetilde{a}_{2},...\} $, to $ X $ marking off $ A $, and then adding new isolated points marking off $ B $ and the $ \widetilde{a}_{n} $'s. Note that $ B \cup \{\widetilde{a}_{1}, \widetilde{a}_{2},...\} $ is closed in $ I(X,A) $, since the limit points of the $ \widetilde{a}_{n} $ are in $ A \subseteq B $. Furthermore, in both steps we have respected the restriction on isolated points imposed in the definition of $ I $, in the first step because $ X $ is perfect, and in the second because the only isolated points are the $ \widetilde{a}_{n} $. Therefore, by repeated applications of Proposition \ref{isolated points}:

If $ f: X \to Y $ is a homeomorphism with $ f[A] = C $ and $ f[B] = D $, then $ f $ extends to a homeomorphism $ I(X,A) \to I(Y,C) $. Since this must send isolated points to isolated points, it must take $ B \cup \{\widetilde{a}_{1},\widetilde{a}_{2},...\} $ to $D \cup \{\widetilde{c}_{1},\widetilde{c}_{2},...\} $. Therefore it extends to a homeomorphism $ I_{2}(X,B,A) \to I_{2}(Y,D,C) $.

Conversely, if $ g: I_{2}(X,B,A) \to I_{2}(Y,D,C) $ is a homeomorphism, it restricts to a homeomorphism $ I(X,A) \to I(Y,C) $ with $ f[B \cup \{\widetilde{a}_{1},\widetilde{a}_{2},...\}] = D \cup \{\widetilde{c}_{1},\widetilde{c}_{2},...\} $. Further, this restricts to a homeomorphism $ X \to Y $ taking $ f[A] = C $. But then $ f[B] = f[(B \cup \{\widetilde{a}_{1},\widetilde{a}_{2},...\}) \cap X] = (D \cup \{\widetilde{c}_{1},\widetilde{c}_{2},...\}) \cap Y = D $, and $ (X,B,A) \cong_{(1,1)} (Y,D,C) $. So $ I_{2} $ is a reduction to $ \approx $.
\end{proof}

\bibliographystyle{amsalpha}
\bibliography{HomeomorphismUniversality}

\end{document}